\documentclass[12pt,reqno]{amsart}

\usepackage{amssymb}
\usepackage{amsmath}
\usepackage{latexsym}
\usepackage{amsthm}
\usepackage{psfig}

\theoremstyle{plain}
\newtheorem{theorem}{Theorem}[section]
\newtheorem{proposition}{Proposition}[section]

\newtheorem{lemma}{Lemma}[section]
\newtheorem{remark}{\it Remark}[section]
\theoremstyle{definition}

\newcommand{\rot}{\mathop\mathrm{rot}}
\newcommand{\grad}{\mathop\mathrm{grad}}
\renewcommand{\div}{\mathop\mathrm{div}}
\newcommand{\Tr}{\mathop\mathrm{Tr}}
\renewcommand{\Re}{\mathop\mathrm{Re}}

\title
[   Lieb--Thirring inequalities on some manifolds ]
 { Lieb--Thirring inequalities on some  manifolds}
\author[A.A.Ilyin]{Alexei A.~Ilyin}
\address
{Keldysh Institute of Applied Mathematics
}
\email{ilyin@keldysh.ru}

\begin{document}

\maketitle

\medskip

\bigskip
\begin{quote}{\normalfont\fontsize{8}{10}\selectfont
{\bfseries Abstract.}
We prove Lieb--Thirring inequalities with improved
constants on the two-dimensional
sphere $\mathbb{S}^2$ and the two-dimensional torus
$\mathbb{T}^2$. In the one-dimensional periodic case
we obtain a simultaneous bound for the negative trace and
the number of negative eigenvalues.

\medskip



\par}
\end{quote}

\setcounter{equation}{0}
\section{Introduction}\label{S:Intro}
The Schr\"odinger operator in $L_2(\mathbb{R}^n)$
$$
-\Delta+V
$$
with a real-valued  potential $V$ that sufficiently fast decays at infinity
has a discrete  negative
spectrum satisfying the Lieb--Thirring
spectral inequalities~\cite{LT}

\begin{equation}\label{est L-T}
\sum_{\nu_i\le0}|\nu_i|^\gamma\le\mathrm{L}_{\gamma,n}
 \int V_-(x)^{\gamma+ n/2}dx,
\end{equation}
where $V_\pm(x)=(|V(x)|\pm V(x))/2$. The Lieb--Thirring constant  $\mathrm{L}_{\gamma,n}$
is finite for $\gamma\ge 1/2$, $n=1$ (for $\gamma=1/2$
see~\cite{Weidl}); $\gamma>0$, $n=2$; and $\gamma\ge0$,
$n\ge3$ (where $\gamma=0$ is the Lieb--Cwikel--Rozenblum
inequality).

The Lieb--Thirring constants satisfy the lower bound
\begin{equation}\label{L1nclass}
\mathrm{L}_{\gamma,n}\ge \mathrm{L}_{\gamma,n}^{\mathrm{cl}}=
\frac1{(2\pi)^n}\int_{\mathbb{R}^n}(1-|\xi|)_+^\gamma dx=
\frac{\Gamma(\gamma+1)}
{(4\pi)^{n/2}\Gamma(n/2+\gamma+1)}\,.
\end{equation}
Sharp results valid for all dimensions $n$, $\mathrm{L}_{\gamma,n}=
\mathrm{L}_{\gamma,n}^{\mathrm{cl}}$, $\gamma\ge3/2$ were
obtained in~\cite{Lap-Weid} (see also~\cite{B-L}).
 The best known estimate of $\mathrm{L}_{\gamma,n}$
for $1\le\gamma<3/2$ from \cite{D-L-L} is as follows
\begin{equation}\label{R=1.81}
\mathrm{L}_{\gamma,n}\le
R\cdot\mathrm{L}_{\gamma,n}^{\mathrm{cl}},
\quad R=\frac\pi{\sqrt{3}}=1.8138\dots
\end{equation}
and improves the previous result~\cite{H-L-W}: $R=2$.

The spectral inequality~(\ref{est L-T}) for the negative trace
(that is, for $\gamma=1$) is equivalent to the following
integral inequality for orthonormal families.
Let $\{\varphi_j\}_{j=1}^N\in H^1(\mathbb{R}^n)$ be an
orthonormal family in $L_2(\mathbb{R}^n)$. Then
$\rho(x):=\sum_{j=1}^N\varphi_j(x)^2$ satisfies the
inequality
\begin{equation}\label{L-T-orth-orig}
\int\rho(x)^{1+2/n}dx\le
\mathrm{k}_{n}\sum_{j=1}^N\|\nabla\varphi_j\|^2,
\end{equation}
where the best constants $\mathrm{k}_{n}$ and  $\mathrm{L}_{1,n}$
satisfy \cite{LT}, \cite{Lieb}
\begin{equation}\label{k=L}
\mathrm{k}_{n}=(2/n)(1+n/2)^{1+2/n}\mathrm{L}_{1,n}^{2/n}.
\end{equation}

In addition to the initial quantum mechanical applications
inequality~(\ref{L-T-orth-orig}) is very important in the
theory of infinite dimensional dynamical systems,
especially, for the attractors of the Navier--Stokes equations
(see, for instance, \cite{Lieb}, \cite{B-V}, \cite{Ch-V-book},
\cite{CF88}, \cite{T} and the references therein).
Accordingly, for satisfying these needs
Lieb-Thirring inequalities~(\ref{L-T-orth-orig}) were
generalized to higher-order elliptic operators on domains
with various boundary conditions and Riemannian manifolds
\cite{G-M-T}, \cite{T}. However, no information was
available on the values of the corresponding constants.
A different approach to the Lieb-Thirring inequalities for
periodic functions, based on the methods of trigonometric
series, was proposed in~\cite{Kashin}.

In this article we shall be dealing with Lieb--Thirring
inequalities on manifolds. We consider the two-dimensional
torus $\mathrm{T}^2=[0,2\pi]^2$ (with flat metric)
and the two-dimensional sphere $\mathbb{S}^2$.
Below we denote by  $M$ either $\mathrm{T}^2$ or
 $\mathrm{S}^2$.
Both the scalar  and  vector-functions are considered.
We first observe that for scalar functions
inequality~(\ref{L-T-orth-orig}) cannot hold unless we
somehow get rid of the constants, and we assume that
the $\varphi_j$'s satisfy
\begin{equation}\label{int=0}
\int_M\varphi dM=0.
\end{equation}
Accordingly,
the Schr\"odinger operator is of the form
\begin{equation}\label{Int_Sch}
-\Delta\varphi+\Pi(V\varphi),
\quad\text{where}\quad
\Pi f=f-\frac1{|M|}\int_{M}fdM,
\end{equation}
and $|M|$ denotes the measure of $M$.
In section~\ref{S:S2T2} we  obtain a bound for the negative
trace of the operator~(\ref{Int_Sch}) on $M$
$$
 \sum_{\nu_j\le0}|\nu_j|\le
   \mathrm{L}_1(M)\int_{M}V_-(x)^2dM\quad\text{with}\quad
\mathrm{L}_{1}(M)\le\frac38.
$$
It is worth pointing out that we obtain the same bound as
in the original paper~\cite{LT}  for the constant
$\mathrm{L}_{1,2}(\mathbb{R}^2)$. As in~\cite{LT} we use
the Birman--Schwinger  kernel (see also~\cite{T}). The
current best known results~(\ref{R=1.81}) for
$\mathbb{R}^n$ are, of course, much sharper. However, the
argument in~\cite{AL} and  induction in the
dimension~\cite{D-L-L}, \cite{Lap-Weid}, \cite{H-L-W} are
not directly  applicable to the case of the torus and the
sphere because of the global condition(\ref{int=0})
(especially since on the sphere
 there is no global coordinate system without singular points).

Next, we consider the case of vector-functions
and show that
\begin{equation}\label{vector}
\mathrm{L}_{1}^{\mathrm{vec}}(M)\le\frac34.
\end{equation}
This is, of course, obvious for the torus since the vector
Laplacian acts independently on the two components of
vector-functions. This is not the case for the sphere,
but~(\ref{vector}) still holds. We also observe that
for the sphere (as for any simply connected manifold)
we do not need any orthogonality conditions
and the (negative) vector Laplacian is strictly positive on
$\mathbb{S}^2$. Using the one-to-one
correspondence between divergence-free and potential vector
fields inherent in two dimensions we show that in the
divergence-free case the bound for the corresponding
Lieb--Thirring constant is the same as in the scalar case.
Finally, in the three-dimensional case we prove the
inequality for the negative trace for $\mathbb{T}^3$ with the original
Lieb--Thirring constant $\frac4{15\pi}$~\cite{LT}
and some $1.039\%$ larger constant for $\mathbb{S}^3$.

In section~\ref{S:1D} we consider the one-dimensional case.
Using the idea of C.\,Foias \cite[p.\,440]{T}
(see also \cite{E-F}) and a recent refinement~\cite{Zelik} of the
multiplicative inequality characterizing the
imbedding
$\dot{H}^1(\mathbb{S}^1)\hookrightarrow L_\infty(\mathbb{S}^1)$
we obtain for the operator
$$
-\frac{d^{2}\varphi}{dx^{2}}+\Pi(V\varphi),
$$
acting on $2\pi$-periodic functions with mean value zero
 the following simultaneous bound for the negative
trace and the number $N$ of negative eigenvalues:
$$
\sum_{j=1}^N|\nu_j|+N\frac1{\pi^2}\le
\frac{2}{3\sqrt{3}}
\int_0^{2\pi}V(x)_-^{3/2}dx.
$$

In section~\ref{S:Aux} we prove two main technical results
concerning a series and a 2D lattice sum depending on a
parameter. Corresponding to these sums in $\mathbb{R}^n$
are the integrals depending on a parameter which are easily calculated by
scaling. The previous (knowingly non-sharp) estimates for
these sums in~\cite{ILMS93}, \cite{I-M-T} give,
respectively,
$\mathrm{L}_{1}(\mathbb{S}^2)\le1/2$
and
$\mathrm{L}_{1}(\mathbb{T}^2)\le 3/({2\pi})$.

In conclusion we  recall the basic facts concerning the Laplace
operator on the sphere~\cite{S-W}.
Let $\mathbb{S}^{m-1}$ be the $(m-1)$-dimensional sphere. We have
for the (scalar) Laplace-Beltrami operator
$\Delta=\div\grad$:
$$
-\Delta Y_n^k=\Lambda_n Y_n^k,\quad
k=1,\dots,k_m(n),\quad n=1,2,\dots.
$$
Here the $Y_n^k$ are the orthonormal spherical harmonics.
Each eigenvalue
$$
\Lambda_n=n(n+m-2)
$$
has multiplicity
$$
k_m(n)=\frac{2n+m-2}{n}\binom{n+m-3}{n-1}.
$$
For example, for $m=2,3,4$ we have
\begin{equation}\label{eig_mult}
\aligned
&\mathbb{S}^1:\ \Lambda_n=n^2,\ k_2(n)=2,\\
&\mathbb{S}^2:\ \Lambda_n=n(n+1),\ k_3(n)=2n+1,\\
&\mathbb{S}^3:\ \Lambda_n=n(n+2),\ k_4(n)=(n+1)^2.
\endaligned
\end{equation}
The following identity is essential \cite{S-W}:
for any $s\in\mathbb{S}^{m-1}$
\begin{equation}\label{identity}
\sum_{l=1}^{k_m(n)}Y_n^l(s)^2=\frac{k_m(n)}{\sigma(m)},
\end{equation}
where $\sigma(m)=2\pi^{m/2}/\Gamma(m/2)$ is the surface area of $S^{m-1}$. In the
vector case we have the similar identity for the gradients
of spherical harmonics \cite{I93}: for any $s\in\mathbb{S}^{m-1}$
\begin{equation}\label{identity-vec}
\sum_{l=1}^{k_m(n)}|\nabla Y_n^l(s)|^2=\Lambda_n\frac{k_m(n)}{\sigma(m)},
\end{equation}

We also use the following notation labelling the
eigenfunctions and the corresponding eigenvalues with a
single subscript
\begin{equation}\label{single}
-\Delta\varphi_i=\lambda_i\varphi_i,
\end{equation}
where
$$
\{\varphi_i\}_{i=1}^\infty=\{
Y_n^1,\dots,Y_n^{k_m(n)}\}_{n=1}^\infty,\quad
\{\lambda_i\}_{i=1}^\infty=
\underset{\!k_m(n)\ \text{times}}{\{\Lambda_n,
\dots,\Lambda_n\}_{n=1}^{\infty}}.
$$

\setcounter{equation}{0}
\section{Lieb--Thirring inequalities on the sphere
and on the torus}\label{S:S2T2}

In this section we obtain estimates for the negative trace
of the Schr\"odinger operators on the $2D$ sphere $\mathbb{S}^2$
and the $2D$ torus $\mathbb{T}^2=[0,2\pi]^2$. Both cases
are treated simultaneously and we denote below by ${M}$ one
of these manifolds. With a slight abuse of notation a
generic point $x\in\mathbb{T}^2$ and $s\in\mathbb{S}^2$
is denoted by $x$.

For  $V\in L_2(M)$ we consider the quadratic form on
$\dot{H}^1(M)$
\begin{equation}\label{Qf}
    Q_V(h)=\|\nabla h\|^2+\int_{M} V(x) h(x)^2dM,\qquad h\in
    \dot{H}^1(M).
\end{equation}
Here and in what follows $\dot{H}^1(M)$ denotes the
subspace of the Sobolev space ${H}^1(M)$ of functions orthogonal to
constants. The form~(\ref{Qf})  is bounded from below
 and defines the self-adjoint   Schr\"odinger-type  operator
\begin{equation}\label{Schr}
    -\Delta h+\Pi (V h),\quad h\in \dot{H}^1(M)
\end{equation}
with discrete spectrum $\nu_1\le\nu_2\le\dots\to\infty$ accumulating
at infinity.

We estimate the negative trace of~(\ref{Schr}) for
$M=\mathbb{S}^2$ and $M=\mathbb{T}^2$
\begin{equation}\label{LTtrace}
   \sum_{\nu_j\le0}|\nu_j|\le
   \mathrm{L}_1(M)\int_{M}V_-(x)^2dM.
\end{equation}

\begin{theorem}\label{T:L1-scal}
For $M=\mathbb{S}^2$ and $M=\mathbb{T}^2$
\begin{equation}\label{const_L1T2}
\mathrm{L}_1(\mathbb{T}^2)<\frac38,\qquad
\mathrm{L}_1(\mathbb{S}^2)<\frac3{8}\,.
\end{equation}
\end{theorem}
\begin{proof}
As usual we first assume that the potential $V$ is smooth.
Having proved~(\ref{LTtrace}) for smooth $V$ we prove the
general case by approximating $V$ with smooth potentials
$V_n$.
We denote by $N_r(V)$ the number of eigenvalues $\nu_j$ such that
 $\nu_j\le r$. Then
\begin{equation}\label{Nr}
\sum_{\nu_j\le0}|\nu_j|^\gamma=\gamma\int_0^\infty r^{\gamma-1}N_{-r}(V)dr.
\end{equation}
We use  the Birman--Schwinger inequality
(see~\cite[Appendix, Proposition 2.1]{T}, where this inequality
is adapted to the Schr\"odinger-type operators defined on subspaces).
Setting $g(x)=(V(x)+(1-t)r)_-$, we have
$$
N_{-r}(V)\le\Tr\bigl[g^{1/2}
(\Pi(-\Delta+tr)\Pi)^{-1}g^{1/2}\bigr]^k,\ \ r>0,\ k\ge1,\
t\in[0,1],
$$
where the trace is calculated in $L_2(M)$.
Next we use  the convexity
inequality of Lieb and
Thirring  \cite{A}, \cite{LT}: for positive operators $A$ and
$C$,
$\Tr (A^{1/2}CA^{1/2})^k\le\Tr A^{k/2}C^kA^{k/2}$. We obtain
$$
N_{-r}(V)\le\Tr\bigl[g^{k/2}
(\Pi(-\Delta+tr)\Pi)^{-k}g^{k/2}\bigr]=\Tr[ g^{k}(\Pi(-\Delta+tr)\Pi)^{-k}],
$$
where the last equality holds for $k>1$, since in this case
the operator $(\Pi(-\Delta+tr)\Pi)^{-k}$ is of trace class
(and multiplication by $g^{k/2}$ is  bounded  in $L_2(M)$).

Now we show that for  $k>1$
($k=3/2$),
\begin{equation}\label{N-r}
N_{-r}(V)\le\frac{1}{4\pi}\frac1{k-1}(tr)^{1-k}
\int_{M}(V(x)+(1-t)r)_-^kdM.
\end{equation}
We first consider the case $M=\mathbb{S}^2$.
Using
the basis~(\ref{single}) and identity
(\ref{identity}), we have
$$
\aligned &\Tr [ g^{k}(\Pi(-\Delta+tr)\Pi)^{-k}]=
\sum_{j=1}^\infty(g^k(-\Delta+tr)^{-k}\varphi_j,\varphi_j)\\=
&\sum_{j=1}^\infty(\lambda_j+tr)^{-k}\int_{\mathbb{S}^2}
g(s)^k\varphi_j(s)^2dS
\\=
&\sum_{n=1}^\infty(\Lambda_n+tr)^{-k}\int_{\mathbb{S}^2}g(s)^k
\sum_{l=1}^{2n+1}\left( Y_n^l(s)\right)^2dS\\= &
\frac1{4\pi}\sum_{n=1}^\infty\frac{2n+1}{(n(n+1)+tr)^{k}}
\int_{\mathbb{S}^2}g(s)^kdS,
\endaligned
$$
which proves (\ref{N-r}) for $M=\mathbb{S}^2$ in view of
 Proposition~\ref{P:S2}.

For the torus $\mathrm{T}^2$ we use the orthonormal basis
$(2\pi)^{-1}e^{imx}$, $m\in\mathbb{Z}^2_0=\mathbb{Z}^2\setminus 0$ and obtain
$$
N_{-r}(V)\le%\Tr B=
\frac1{4\pi^2}
\sum_{m\in\mathbb{Z}^2_0}\frac1{(|m|^2+tr)^{k}}
\int_{\mathbb{T}^2}g(x)^{k}dx,
$$
which proves (\ref{N-r}) for $M=\mathbb{T}^2$ in view of
 Proposition~\ref{P:T2}.

Next, restricting $k$ to $k\in(1,2)$ and using~(\ref{Nr})
with $\gamma=1$ we have
$$
\sum_{\nu_j\le0}|\nu_j|\le
\frac1{4\pi}\frac1{k-1}\int_{M}\int_0^\infty(tr)^{1-k}(V(x)+(1-t)r)_-^kdrdx.
$$
We evaluate the inner integral  setting
$r=\frac1{1-t}V_-(x)\,\rho. $
 If $V\le0$ and $V_-=-V$, then $(V(x)+(1-t)r)_-=V_-(x)(\rho-1)_-$ and
$$
\int_0^\infty (tr)^{1-k}(V(x)+(1-t)r)_-^kdr=
t^{1-k}(1-t)^{k-2}B(2-k,1+k)V_-(x)^2.
$$
For the optimal $t=k-1\in(0,1)$ we obtain
\begin{equation}\label{LTg1}
\sum_{\nu_j\le0}|\nu_j|\le
\frac1{4\pi}\frac1{k-1}\frac{B(2-k,1+k)}{(k-1)^{k-1}(2-k)^{2-k}}
\int_{M}V_-(x)^2dM,\quad k\in(1,2),
\end{equation}
which proves~(\ref{LTtrace}) with
$$
\mathrm{L}_1(M)\le\frac1{4\pi}
\frac{B(2-k,1+k)}{(k-1)^{k}(2-k)^{2-k}}\,\biggl|_{k=3/2}=\frac38.
$$
The minimum is  attained at $k=1.38\dots$, giving
 $\mathrm{L}_1(M)\le0.3605\,.$
\end{proof}

We now consider the vector case important for
applications. The case $M=\mathbb{T}^2$ involves no
difficulties since the Laplacian acts independently on
the components of a vector field, so we consider
$M=\mathbb{S}^2$. The Laplace operator acting on (tangent) vector
fields on $\mathbb{S}^2$ we define as the Laplace--de Rham
operator $-d\delta-\delta d$ identifying $1$-forms and
vectors. Then for a two-dimensional manifold we have
\cite{I93}
$$
\mathbf{\Delta} u=\nabla\div u-\rot\rot u,
$$
where the operators $\nabla=\grad$ and $\div$ have the
conventional meaning. The operator $\rot$ of a vector $u$ is a
scalar  and for a scalar $\psi$,
$\rot\psi$ is a vector:
$$
\rot u:=-\div(n\times u),\qquad
\rot\psi:=-n\times\nabla\psi,
$$
where $n$ is the unit outward normal vector. We note that for the
operators $\rot$ so defined,  for a scalar $\psi$ it holds
\begin{equation}\label{rotrot}
\rot\rot\psi=-\Delta\psi\ (=-\div\grad\psi).
\end{equation}
 Integrating by parts, that is, using
$$
(\nabla\psi,u)_{L_2(T\mathbb{S}^2)}=-(\psi,\div
u)_{L_2(\mathbb{S}^2)},\quad
(\rot\psi,u)_{L_2(T\mathbb{S}^2)}=(\psi,\rot
u)_{L_2(\mathbb{S}^2)},
$$
we obtain
$$
(-\mathbf{\Delta} u,u)_{L_2(T\mathbb{S}^2)}=\|\rot u\|^2+\|\div u\|^2.
$$
Next, we have the orthogonal sum
$L_2(T\mathbb{S}^2)=H\oplus H^\perp$:
$$
H=\{u\in L_2(T\mathbb{S}^2),\ \div u=0\},\ H^\perp=\{u\in
L_2(T\mathbb{S}^2),\ \rot u=0\}.
$$
Both $H$ and $H^\perp$ are invariant with
respect to $\mathbf{\Delta}$ (in then sense that if $u\in
H$ and $\mathbf{\Delta} u\in  L_2(T\mathbb{S}^2)$, then
$\mathbf{\Delta} u\in H$, and similarly for $H^\perp$) and
there exist two orthonormal systems of
eigenvectors: $\{w_j\}_{j=1}^\infty \in H$ and
$\{v_j\}_{j=1}^\infty \in H^\perp$ with the same
eigenvalues
\begin{equation}\label{bas-vec}
-\mathbf{\Delta} w_j=\lambda_j w_j,\qquad -\mathbf{\Delta}
v_j=\lambda_j v_j,
\end{equation}
where
$$
 w_j=\lambda_j^{-1/2}n\times\nabla\varphi_j,\qquad
v_j=\lambda_j^{-1/2}\nabla\varphi_j.
$$
 Here the
$\lambda_j$'s and the $\varphi_j$'s are the eigenvalues and
eigenfunctions of the scalar Laplacian on $\mathbb{S}^2$,
see~(\ref{single}). Both~(\ref{bas-vec}), and the orthonormality of
the $w_j$'s and $v_j$'s follow from~(\ref{rotrot}). Hence, corresponding
to the eigenvalue
$\Lambda_n=n(n+1)$ there are two families of $2n+1$
orthonormal eigenvectors $w_n^l(s)$ and $v_n^l(s)$,
$l=1,\dots,2n+1$ and~(\ref{identity-vec}) gives the
following important identities: for any $s\in\mathbb{S}^2$
\begin{equation}\label{id-vec}
\sum_{l=1}^{2n+1}|w_n^l(s)|^2=\frac{2n+1}{4\pi},\qquad
\sum_{l=1}^{2n+1}|v_n^l(s)|^2=\frac{2n+1}{4\pi}.
\end{equation}
We finally observe that $-\mathbf{\Delta}\ge \Lambda_1I=2I$.

Having done these preliminaries we consider the quadratic
form
\begin{equation}\label{form-vec}
    Q^\mathrm{vec}_V(u)=\|\rot u \|^2+
    \|\div u \|^2+\int_{\mathbb{S}^2} V(s) |u(s)|^2dS,\quad u\in
  H^1(T\mathbb{S}^2),
\end{equation}
which is bounded from below, and defines the self-adjoint
  Schr\"odinger  operator
$$
    -\mathbf{\Delta} u+ V u
$$
with discrete spectrum. We estimate its negative trace
\begin{equation}\label{LTtrace-vec}
   \sum_{\nu_j\le0}|\nu_j|\le
   \mathrm{L}^\mathrm{vec}_1(\mathbb{S}^2)\int_{\mathbb{S}^2}V_-(s)^2dS.
\end{equation}

\begin{theorem}\label{T:L1-vec}
\begin{equation}\label{const_L1T2-vec}
\mathrm{L}^\mathrm{vec}_1(\mathbb{S}^2)\le\frac3{4}.
\end{equation}
\end{theorem}
\begin{proof}
Using the basis~(\ref{bas-vec}), identity (\ref{id-vec}),
similarly to Theorem~\ref{T:L1-scal}
$$
\aligned  &N_{-r}(V)\le\Tr[
g^{k}(-\mathbf{\Delta}+tr)]^{-k}
\\=
 &\sum_{j=1}^\infty(g^k(-\mathbf{\Delta}+tr)^{-k}w_j,w_j)+
\sum_{j=1}^\infty(g^k(-\mathbf{\Delta}+tr)^{-k}v_j,v_j)
\\=
&2\frac1{4\pi}\sum_{n=1}^\infty\frac{2n+1}{(n(n+1)+tr)^{k}}
\int_{\mathbb{S}^2}g(s)^kdS\le
\frac{1}{2\pi}\frac1{k-1}(tr)^{1-k}
\int_{\mathbb{S}^2}g(s)^kdS,
\endaligned
$$
and we complete the proof as in Theorem~\ref{T:L1-scal}.
\end{proof}
\begin{remark}
{\rm
The same estimate holds for the torus
\begin{equation}\label{const_L1T2-vecT2}
\mathrm{L}^\mathrm{vec}_1(\mathbb{T}^2)\le\frac3{4}.
\end{equation}
However, in this case we have to assume that
$u$ has zero average.
}
\end{remark}

Spectral inequalities   (\ref{LTtrace}) and~(\ref{LTtrace-vec}) are
 equivalent to the
integral inequalities for families of orthonormal functions
and vector fields. As before, $M$ stands for $\mathbb{S}^2$
or~$\mathbb{T}^2$.

\begin{theorem}
 Let $\{\varphi_j\}_{j=1}^N\in\dot{H}^1(M)$
be an orthonormal scalar family. Then for
$\rho(x):=\sum_{j=1}^N\varphi_j(x)^2$ the following
inequality~holds:
\begin{equation}\label{orthscal}
\int_{M}\rho(x)^2dM\le \mathrm{k}_2\sum_{j=1}^N\|\nabla \varphi_j\|^2,
\quad \mathrm{k}_2\le\frac32\,.
\end{equation}

If a family of vector fields $\{u_j\}_{j=1}^N\in H^1(TM)$
is orthonormal in $L^2(TM)$,  then
\begin{equation}\label{orthvec}
\int_{M}\rho(x)^2dM\le
 \mathrm{k}_2^{\mathrm{vec}}\sum_{j=1}^N(\|\rot u_j\|^2
+\|\div u_j\|^2),
\quad \mathrm{k}_2^{\mathrm{vec}}\le 3,
\end{equation}
where $\rho(x)=\sum_{j=1}^N|u_j(x)|^2$.
If, in addition, $\div u_j=0$ $($or $\rot u_j=0$$)$
for $j=1,\dots,N$,  then
\begin{equation}\label{orthvecsol}
\int_{M}\rho(x)^2dM\le
\begin{cases}\displaystyle
\mathrm{k}_2^{\textrm{\rm sol}}\sum_{j=1}^N\|\rot u_j\|^2,
\quad \ \div u_j=0,
\\\displaystyle
\mathrm{k}_2^{\textrm{\rm pot}}\sum_{j=1}^N\|\div u_j\|^2,
\quad \rot u_j=0,
\end{cases}
\end{equation}
 where
\begin{equation}\label{vec-const}
\mathrm{k}_2^{\textrm{\rm sol}}=
\mathrm{k}_2^{\textrm{\rm pot}}\le
\frac{\mathrm{k}_2^{\mathrm{vec}}}2\le\frac32\,.
\end{equation}
\end{theorem}
\begin{proof}
In two dimensions the  relation~(\ref{k=L}) between the constants
$\mathrm{k}_2$ and $\mathrm{L}_1$
is as follows (the fact that we are dealing with manifolds
does not play a role)
\begin{equation}\label{k2=L2}
\mathrm{k}_2=4\mathrm{L}_1.
\end{equation}
This proves~(\ref{orthscal}) and~(\ref{orthvec}).
For the sake of completeness we recall the proof of
(\ref{orthvecsol}), (\ref{vec-const}) from~\cite{I-M-T}. By symmetry
inherent in the two-dimensional case
$$
\div u=0\Leftrightarrow \rot\widehat u=0,
\quad\text{where}\quad \widehat u=n\times u.
$$
Furthermore,  $u_1,\dots,u_N$ are orthonormal  if and only
if
$\widehat u_1,\dots,\widehat u_N$ are orthonormal. This shows that
$\mathrm{k}_2^{\textrm{\rm sol}}=
\mathrm{k}_2^{\textrm{\rm pot}}$.
Let us prove the inequality
$\mathrm{k}_2^{\textrm{\rm
sol}}\le\frac{\mathrm{k}_2^{\mathrm{vec}}}2$.
 Let $u_1,\dots,u_N$
be orthonormal  and let $\div u_j=0$,
$j=1,\dots,N$. We set $\rho(x)=\sum_{j=1}^N|u_j(x)|^2$ and
consider the family of $2N$~vector functions
$u_1,\dots,u_N,\widehat u_1,\dots,\widehat u_N$. Since
$\div u_j=0$ and
$\rot\widehat u_j=0$, $j=1,\dots N$, we have $(u_i,\widehat u_j)=0$
for $1\le i,j\le N$, and the whole family is orthonormal.
Applying~(\ref{orthvec}) to this family of
$2N$~functions and taking into account that
$|u_j(x)|=|\widehat u_j(x)|$ and $\div\widehat u_j(x)=-\rot u_j(x)$ we obtain
$$
\aligned
4\int_M\rho(x)^2\,dx
&=\int_M\biggl(\,\sum_{j=1}^N\bigl(|u_j(x)|^2+
|\widehat u_j(x)|^2\bigr)\biggr)^2\,dx\le
\\
&\le\mathrm{k}_2^{\mathrm{vec}}
\sum_{j=1}^N\bigl(\|{\rot u_j}\|^2+\|{\div\widehat u_j}\|^2\bigr)
=2\mathrm{k}_2^{\mathrm{vec}}\sum_{j=1}^N\|{\rot u_j}\|^2.
\endaligned
$$
Therefore
$\mathrm{k}_2^{\textrm{\rm
sol}}\le{\mathrm{k}_2^{\mathrm{vec}}}/2\le3/2$.
\end{proof}
\begin{remark}\label{R:lower_bound}
{\rm
The lower bound for $\mathrm{k}_2(M)$ is the same as
in $\mathbb{R}^2$
\begin{equation}\label{lower_bound_M}
\mathrm{k}_2(M)\ge\frac1{2\pi}.
\end{equation}
For instance, for the sphere we take the first $N$
eigenfunctions~(\ref{single}) and use the fact that
$\lambda_j=[j^{1/2}]([j^{1/2}]+1)\sim j$. Then
$$
N^2=\biggl(\int_{\mathbb{S}^2}\rho(s)dS\biggr)^2\le
4\pi\|\rho\|^2\le 4\pi\mathrm{k}_2\sum_{j=1}^N\lambda_j\sim
2\pi \mathrm{k}_2 N^2.
$$
Accordingly, in view of (\ref{k2=L2}),
$$
\mathrm{L}_1(M)\ge\frac1{8\pi}.
$$
The same lower bound holds for $\mathbb{T}^2$ since in this
case $\lambda_j\sim j/\pi$.
}
\end{remark}

Concluding this section we briefly consider the
three-dimensional case.  For $\mathbb{S}^3$ we see
from~(\ref{eig_mult}) that the eigenvalue $\Lambda_n=n(n+2)$
has multiplicity $(n+1)^2$ and arguing as in Theorem~\ref{T:L1-scal}
and setting $k=2$ we obtain using Proposition~\ref{P:T3S3}
$$
\aligned
N_{-r}(V)\le\frac1{2\pi^2} \sum_{n=1}^\infty
\frac{(n+1)^2}{(n(n+2)+tr)^2}\int_{\mathbb{S}^3}g(s)^2dS\\
\le\frac{\delta_{\mathbb{S}^3}}{8\pi}(tr)^{-1/2}\int_{\mathbb{S}^3}g(s)^2dS.
\endaligned
$$
For the torus $\mathrm{T}^3$ using the basis of
exponentials
$(2\pi)^{-3/2}e^{imx}$, $m\in\mathbb{Z}^3_0$ we have
$$
\aligned
 N_{-r}(V)\le \frac1{8\pi^3}
\sum_{m\in\mathbb{Z}^3_0}\frac1{(|m|^2+tr)^2}
\int_{\mathbb{T}^2}g(x)^2dx\\<
\frac{\delta_{\mathbb{T}^3}}{8\pi}(tr)^{-1/2}\int_{\mathbb{T}^3}g(x)^2dx.
\endaligned
$$
We set  $t=1/2$ and for a fixed $x\in M$ calculate the integral
$$
\int_0^\infty
(tr)^{-1/2}(V(x)+(1-t)r)_-^2dr=\frac{32}{15}V_-(x)^{5/2}
$$
and obtain using~(\ref{Nr}) the following result.

\begin{theorem}
The negative spectrum of the operator
$-\Delta +\Pi(V\cdot)$ on $M=\mathbb{S}^3$ or $\mathbb{T}^3$
satisfies
$$
   \sum_{\nu_j\le0}|\nu_j|\le
   \mathrm{L}_1(M)\int_{M}V_-(x)^{5/2}dM,
$$
where
$$
\mathrm{L}_1(M)\le\delta_M\frac{4}{15\pi}\,.
$$
Here $\delta_{\mathbb{S}^3}=1.0139\dots$ and
 $\delta_{\mathbb{T}^3}=1$.
\end{theorem}

\setcounter{equation}{0}
\section{One-dimensional two-term Lieb--Thirring inequalities}\label{S:1D}

The  imbedding of the Sobolev space
$H^l(\mathbb{R})$, $l>1/2$, into the space of bounded continuous
functions can be written in the form of a multiplicative
inequality
\begin{equation}\label{my}
\|f\|_\infty^2\le
c(l)\|f\|^{2-1/l}\|f^{(l)}\|^{1/l},
\end{equation}
where the sharp constant $c(l)$ was found in~\cite{Taikov}:
\begin{equation}\label{Tai}
c(l)=(2l\alpha^\alpha(1-\alpha)^{1-\alpha}\sin\pi\alpha)^{-1},
\quad \alpha=1/({2l}).
\end{equation}
It was also shown there that there exists a unique (up to
dilations and translations) extremal function.
For periodic functions with zero average $f\in\dot{H}^l(\mathbb{S}^1)$
inequality~(\ref{my}) holds with the same sharp
constant~(\ref{Tai}), however, there
are no extremal functions~\cite{I98JLMS}. An important improvement
of~(\ref{my}) for $2\pi$-periodic functions
 has been recently obtained in~\cite{Zelik},
where it was shown that
\begin{equation}\label{Zel}
\|f\|_\infty^2\le
c(l)\|f\|^{2-1/l}\|f^{(l)}\|^{1/l}-K(l)\|f\|^2.
\end{equation}
For all $l$ the constant  $K(l)>0$ and, in particular,
$K(1)=1/\pi$ and $K(2)=2/(3\pi)$, so that
\begin{equation}\label{Zel2}
\|f\|_\infty^2\le 1\cdot\|f\|\|f'\|-\frac1\pi\|f\|^2,\
\|f\|_\infty^2\le
(4/27)^{1/4}\|f\|^{3/2}\|f''\|^{1/2}-\frac2{3\pi}\|f\|^2,
\end{equation}
where all four constants are sharp and no extremal
functions exist.
\begin{theorem}\label{T:L-T-rem-int}
Suppose that
$\{\varphi_j\}_{j=1}^N\subset\dot{H}^l(\mathbb{S}^1)$ is an
orthonormal family in ${L}_2(\mathbb{S}^1)$. Then for
$\rho(x):=\sum_{j=1}^N\varphi_j(x)^2$ the following
inequality holds:
\begin{equation}\label{int-L-T-rem}
\int_0^{2\pi}\rho(x)^{2l+1}dx+N\cdot K(l)^{2l}\le
c(l)^{2l}\sum_{j=1}^N\|\varphi_j^{(l)}\|^2.
\end{equation}
\end{theorem}
\begin{proof} For any
$\xi\in\mathbb{R}^N$ using~(\ref{Zel}) with
$f(x)=\sum_{j=1}^N\xi_j\varphi_j(x)$ we have
$$
\biggl|\!\sum_{j=1}^N\xi_j\varphi_j(x)\!\biggr|^2\!\le
c(l)\biggl(\sum_{j=1}^N\xi_j^2\!\biggr)^\frac{2l-1}{2l}\!
\biggl(\sum_{i,j=1}^N\xi_i\xi_j
\bigl(\varphi_i^{(l)}\!,\varphi_j^{(l)}\bigr)\biggr)^\frac1{2l}\!-\!
K(l)\biggl(\sum_{j=1}^N\xi_j^2\biggr),
$$
by orthonormality.
Setting $\xi_j=\varphi_j(x)$
we obtain
$$
\rho(x)^2\le
c(l)\rho(x)^\frac{2l-1}{2l}
\biggl(\sum_{i,j=1}^N\varphi_i(x)\varphi_j(x)
\bigl(\varphi_i^{(l)}\!,\varphi_j^{(l)}\bigr)\biggr)^\frac1{2l}-
K(l)\rho(x),
$$
or
\begin{multline*}
\rho(x)^{2l+1}+K(l)^{2l}\rho(x)\le\rho(x)\bigl(\rho(x)+K(l)\bigr)^{2l}\\
\le c(l)^{2l}
\sum_{i,j=1}^N\varphi_i(x)\varphi_j(x)
\bigl(\varphi_i^{(l)},\varphi_j^{(l)}\bigr).
\end{multline*}
%$$
%\rho(x)^{2l+1}+K(l)^{2l}\rho(x)\le\rho(x)\bigl(\rho(x)+K(l)\bigr)^{2l}\le
%c(l)^{2l}
%\sum_{i,j=1}^N\varphi_i(x)\varphi_j(x)
%\bigl(\varphi_i^{(l)}\!,\varphi_j^{(l)}\bigr).
%$$
Integrating and again using orthonormality  we finally
obtain~(\ref{int-L-T-rem}).
\end{proof}

For $V(x)\ge0$ we consider the following
quadratic form on $\dot{H}^l(\mathbb{S}^1)$
\begin{equation}\label{form}
\int_0^{2\pi}\varphi^{(l)}(x)^2dx-\int_0^{2\pi}V(x)\varphi(x)^2dx,
\end{equation}
which is bounded from below and defines a
Schr\"odinger-type operator
\begin{equation}\label{oper}
-\frac{d^{2l}\varphi}{dx^{2l}}-\Pi(V\varphi).
\end{equation}
In view of compactness of $\mathbb{S}^1$ the spectrum of
this operator is discrete.
\begin{theorem}\label{T:L-T-rem-spec}
Suppose that there exist $N$ negative eigenvalues $-\nu_j\le0$, %$\nu_j\ge0$,
 $j=1,\dots,N$ of the operator~(\ref{oper}). Then
 both the negative trace and the number $N$ of negative
 eigenvalues satisfy the following inequality
\begin{equation}\label{L-T-rem-spec}
\sum_{j=1}^N\nu_j+N\cdot\biggl(\frac{K(l)}{c(l)}\biggr)^{2l}\le
\frac{2l}{(2l+1)^{\frac{2l+1}{2l}}}\cdot
c(l)\int_0^{2\pi}V(x)^\frac{2l+1}{2l}dx.
\end{equation}
\end{theorem}
\begin{proof}
Let the orthonormal eigenfunctions $\varphi_j(x)$
correspond to the eigenvalues $-\nu_j$. Then
$$
\int_0^{2\pi}\varphi_j^{(l)}(x)^2dx-\int_0^{2\pi}V(x)\varphi_j(x)^2dx=
-\nu_j.
$$
Setting as before $\rho(x):=\sum_{j=1}^N\varphi_j(x)^2$
and  using~(\ref{int-L-T-rem}) we obtain
$$
\aligned \sum_{j=1}^N\nu_j&=\int_0^{2\pi}V(x)\rho(x)dx-
\sum_{j=1}^N\|\varphi_j^{(l)}\|^2\\
&\le \|V\|_{L_{\frac{2l+1}{2l}}}\|\rho\|_{L_{2l+1}}-
\frac1{c(l)^{2l}}\|\rho\|_{L_{2l+1}}^{2l+1}-
N\cdot\biggl(\frac{K(l)}{c(l)}\biggr)^{2l}\\
&\le \max_y\biggl(\|V\|_{L_{\frac{2l+1}{2l}}}y-
\frac1{c(l)^{2l}}y^{2l+1}\biggr)-
N\cdot\biggl(\frac{K(l)}{c(l)}\biggr)^{2l}.
\endaligned
$$
Calculating the maximum we obtain~(\ref{L-T-rem-spec}).
\end{proof}
\begin{remark}
{\rm
It is worth pointing out that unlike $c(l)$, the constants
$K(l)$ are not dimensionless and for $L$-periodic functions
(with mean value zero) we have
$K_L(l)=K_{2\pi}(l)(2\pi/L)$.
For example, for $l=1$ %we have
\begin{equation}\label{L}
\aligned
\int_0^{L}\rho(x)^{3}dx+N \frac4{L^2}&\le
\sum_{j=1}^N\|\varphi'_j\|^2,\\
\sum_{j=1}^N\nu_j+N\frac4{L^2}&\le
\frac{2}{3\sqrt{3}}
\int_0^{L}V(x)^{3/2}dx.
\endaligned
\end{equation}
}
\end{remark}
\begin{remark}
{\rm
If the potential $V$ is even (and periodic), then  the subspace
of odd periodic functions is invariant for the operator
$$
-\frac{d^{2l}}{dx^{2l}}\varphi-V\varphi,
$$
and the orthogonal projection $\Pi$ (\ref{Int_Sch}) can be
omitted.
}
\end{remark}

\section{Auxiliary inequalities}\label{S:Aux}

\begin{proposition}\label{P:S2}
For $\mu\ge0$ and $k=3/2$
\begin{equation}\label{ineqS2}
H(\mu):=
\mu^{2k-2}\sum_{n=1}^\infty\frac{2n+1}{((n(n+1)+\mu^2)^k}<
 \frac{1}{k-1}.
\end{equation}
\end{proposition}
\begin{proof}
Since
\begin{equation}\label{H(mu)}
H(\mu)=\mu^{-2}\sum_{n=1}^\infty(2n+1)f( n(n+1)/\mu^2),
\end{equation}
where
$$
f(x)=\frac1{(x+1)^k}\quad\text{and}
\quad\int_0^\infty f(x)dx=\frac1{k-1},
$$
the fact that inequality~(\ref{ineqS2}) holds for all
$\mu\ge\mu_0$, where $\mu_0$ is sufficiently large,
follows from Lemma~\ref{L:S2} below, which gives the asymptotic
expansion of $H(\mu)$ for large $\mu$:
$$
H(\mu)=\frac1{k-1}-\frac23\frac1{\mu^2}+o(1/\mu^2).
$$
 The point $\mu_0=5.0833$
is specified in the Appendix (see section~\ref{S:Appen}).
On the  {\it finite} interval $[0,\mu_0]$ we make sure that
(\ref{ineqS2}) holds by  numerical calculations. The graph
of $H(\mu)$ on $[0,\mu_0]$ is shown in Fig.~\ref{fig:S2}.
\end{proof}

\begin{lemma}\label{L:S2} Suppose that $f$ is sufficiently smooth
and sufficiently fast decays at infinity. Then the
following asymptotic expansion as $\mu\to\infty$ holds for
$H(\mu)$ defined in~(\ref{H(mu)}):
\begin{equation}\label{asym}
H(\mu)= \int_0^\infty f(x)dx-
\frac1{\mu^2}\frac23f(0)+o(1/\mu^2).
\end{equation}

\end{lemma}
\begin{proof}
We consider the following partitioning of the half-line
$x\ge 0$ by the points
$$
a_n\,=\,a_n(\mu)\,=\,\frac {(n-1)n}{\mu^2},\ \ n=1,\dots\,.
$$
Then a direct inspection shows that
$$
\aligned
\mu^{-2}\sum_{n=1}^\infty n f(n(n+1)/\mu^2)
\,&=\,\frac 12
\sum_{n=1}^\infty f(a_{n+1})(a_{n+1}-a_n),\\
\mu^{-2}\sum_{n=1}^\infty(n+1)f(n(n+1)/\mu^2)
\,&=\,\frac 12
\sum_{n=1}^\infty f(a_{n+1})(a_{n+2}-a_{n+1}).
\endaligned
$$
Therefore
$$
H(\mu)\,=\,
\frac 12 f(a_2)(a_2-a_1)\,+\,\sum\limits_{n=2}^\infty
\frac {f(a_n)+f(a_{n+1})}2\,(a_{n+1}-a_n).
$$
Next, we recall
the trapezoidal formula for the approximate
calculation of the integrals (see, for instance, \cite{Krylov}):
\begin{equation}\label{trap}
\int\limits_a^bf(x)dx=\frac{f(a)+f(b)}2(b-a)+ R_{a,b}(f),
\end{equation}
where $$ R_{a,b}(f)=-\frac {(b-a)^3}{12}\,f^{''}(\xi), \ \
a<\xi<b\,.
$$
 This gives
\begin{equation}\label{int-sum1}
\aligned &\int_0^\infty
f(x)dx=\sum_{n=1}^\infty\int_{a_n}^{a_{n+1}}f(x)dx\\=
 \int_{a_1}^{a_2}f(x)dx&+
\sum_{n=2}^\infty\frac{f(a_n)+f(a_{n+1})}2\,(a_{n+1}-a_n)\,+\,
\sum\limits_{n=2}^\infty R_{a_n,a_{n+1}}(f)\\= &H(\mu)\,+\,
\int_{a_1}^{a_2}(f(x)- f(a_2)/2)dx\,+\, \sum_{n=2}^\infty
R_{a_n,a_{n+1}}(f).
\endaligned
\end{equation}
Since $a_1=0$ and $a_2=2/\mu^2$ we clearly have
$$
\lim_{\mu\to\infty}
\mu^2\int_{a_1}^{a_2}(f(x)-
f(a_2)/2)dx=f(0).
$$
For the third term, using~(\ref{trap})  with
\begin{equation}\label{ksi}
\xi_n\in (a_n,a_{n+1}),\quad
\xi_n=\frac{n^2}{\mu^2} +\frac {\theta_nn}{\mu^2}, \quad
|\theta_n|<1
\end{equation}
we obtain
$$
\aligned \lim_{\mu\to\infty}\mu^2 \sum_{n=2}^\infty
R_{a_n,a_{n+1}}(f)=
%-\frac23\lim_{\mu\to\infty}\frac1\mu\sum_{n=2}^\infty
%(n/\mu)^3f''(\xi_n)=
-\frac23\lim_{\mu\to\infty}\frac1\mu\sum_{n=1}^\infty
(n/\mu)^3f''(\xi_n)\\=
-\frac23\lim_{\mu\to\infty}\frac1\mu\sum_{n=1}^\infty
(n/\mu)^3f''(n^2/\mu^2)= -\frac23\int_0^\infty
x^3f''(x^2)dx= -\frac13f(0),
\endaligned
$$
as the following integration by parts shows:
$$
\aligned
\int_0^\infty x^3f''(x^2)dx=\frac12\int_0^\infty
x^2\left[f'(x^2)\right]'_xdx=%\\
-\int_0^\infty x f'(x^2)dx=
\frac12  f(0).
\endaligned
$$
Thus, the last two terms in~(\ref{int-sum1}) are both of
order $1/\mu^2$ and add up to $\frac2{3\mu^2}f(0)$. The
proof is complete.
\end{proof}
\begin{proposition}\label{P:T2}
For $\mu\ge0$ and $k=3/2$
\begin{equation}\label{ineqT2}
F(\mu):= \mu^{2k-2}\sum_{m\in\mathbb{Z}^2_0}
\frac{1}{(|m|^2+\mu^2)^k}<
 \frac{\pi}{k-1}.
\end{equation}
\end{proposition}
\begin{proof} The function $F(\mu)$ for $k>1$ has the following
asymptotic expansion as $\mu\to\infty$:
\begin{equation}\label{expexp}
F(\mu)=
\frac\pi{k-1}-\frac1{\mu^2}+O(e^{-C\mu}).
\end{equation}
This follows from the the Poisson summation formula (see,
e.\,g., \cite{S-W})
\begin{equation}\label{Poisson}
\sum_{m\in\mathbb{Z}^n}f(m/\mu)=
(2\pi)^{n/2}\mu^n
\sum_{m\in\mathbb{Z}^n}\widehat{f}(2\pi m \mu),
\end{equation}
where
$\mathcal{F}(f)(\xi)=\widehat{f}(\xi)=(2\pi)^{-n/2}\int_{\mathbb{R}^n}
f(x)e^{i\xi x}dx$. For the function $f(x)=1/(1+x^2)^{-k}$,
$x\in\mathbb{R}^2$, this gives
$$
\aligned
F(\mu)=
\frac1{\mu^2}\sum_{m\in\mathbb{Z}^2}f(m/\mu)-\frac1{\mu^2}f(0)=
\frac\pi{k-1}-\frac1{\mu^2}+
2\pi\sum_{m\in\mathbb{Z}^2_0}\widehat{f}(2\pi\mu m).
\endaligned
$$
The third term is exponentially small
as $\mu\to\infty$
since $f$ is analytic in the strip $\Re z_1<a$, $\Re
z_2<a$, $a<\sqrt{2}/2$, and therefore $|\widehat f(\xi)|\le
C(a,k)e^{-a|\xi|}$, see Remark~\ref{R:all_k}.
This proves~(\ref{expexp}). Hence
(\ref{ineqT2}) holds for all $\mu\in[\mu_0,\infty)$.

To specify $\mu_0$ for $k=3/2$  we  take advantage of the  formula
\cite{S-W}:
\begin{equation}\label{Four_Stein}
\mathcal{F}\bigl(1/(1+x^2)^{(n+1)/2}\bigr)(\xi)=
\frac1{c_n(2\pi)^{n/2}}e^{-|\xi|},
\quad x\in\mathbb{R}^n,
\end{equation}
where
$$
\frac1{c_n}=\frac{\pi^{(n+1)/2}}{\Gamma((n+1)/2)}=
\int_{\mathbb{R}^n}\frac{dx}{(1+x^2)^{(n+1)/2}}.
$$
In the two-dimensional case with $k=3/2$
$$
F(\mu)=
\frac\pi{k-1}-\frac1{\mu^2}+2\pi\sum_{m\in\mathbb{Z}^2_0}e^{-2\pi\mu|m|}.
$$
Therefore (\ref{ineqT2}) is equivalent to showing that the
inequality
\begin{equation}\label{asexp}
2\pi\sum_{m\in\mathbb{Z}^2_0}e^{-2\pi\mu|m|}<\frac1{\mu^2}
\end{equation}
holds for all $\mu>0$. To estimate the series on the
right-hand side we write down the numbers
$|m|^2=m_1^2+m_2^2$, $m\in\mathbb{Z}^2_0$, in the
increasing order counting multiplicities and denote them by $\lambda_j$:
$
\{\lambda_j\}_{j=1}^\infty=\{m_1^2+m_2^2,\ m\in\mathbb{Z}^2_0\}.
$
For $\lambda\ge1$ we denote by $N(\lambda)$ the number of
$\lambda_j$'s less than or equal to $\lambda$
(the number of points with integer coordinates inside the circle of
radius $\sqrt{\lambda}$):
$$
N(\lambda)=\sum_{\lambda_j\le\lambda}1.
$$
We inscribe the circle of radius $\sqrt{\lambda}$ into the
square with side $2\sqrt{\lambda}+1$ and cross out the
origin. We obtain
$$
N(\lambda)\le(2\sqrt{\lambda}+1)^2-1=4\lambda+4\sqrt{\lambda}\le
8\lambda.
$$
For $\lambda=\lambda_j$ this gives
$j=N(\lambda_j)\le8\lambda_j$ so that
$\lambda_j\ge\frac j8$.

Returning to~(\ref{asexp}) and setting below
$L:=\pi\mu/2\sqrt{2}$  we have
$$
\aligned
\sum_{m\in\mathbb{Z}^2_0}e^{-2\pi\mu|m|}=\sum_{j=1}^\infty
e^{-2\pi\mu\lambda_j^{1/2}}\le \sum_{j=1}^\infty
e^{-2Lj^{1/2}}= e^{-L}\sum_{j=1}^\infty
e^{-L(2j^{1/2}-1)}
\\
\le
 e^{-L}\sum_{j=1}^\infty
e^{-Lj^{1/2}}< e^{-L}\int_0^\infty e^{-Lx^{1/2}}dx=
\frac{2e^{-L}}{L^2}=\frac{16}{\pi^2\mu^2}e^{-\frac{\pi\mu}{2\sqrt{2}}},
\endaligned
$$
and inequality~(\ref{asexp}) is satisfied for all
$\mu\ge\mu_0=\frac{2\sqrt{2}}\pi\log\frac{32}\pi= 2.0896$.
In fact, $\lambda_j\ge j/4$ (see \cite{I-M-T}), which gives
$\mu\ge\mu_0=\frac{2}\pi\log\frac{16}\pi=1.0363$. On the
{\it finite} interval $[0,\mu_0]$ we verify~(\ref{ineqT2})
on a computer, see Fig.~\ref{fig:S2}.
\end{proof}
\begin{remark}\label{R:all_k}
{\rm
Shifting for $x_1$ and $x_2$ the domain of integration by $\pm ia$
and using analyticity  we obtain
$$
|\widehat
f(\xi)|\le\frac{e^{-a|\xi|}}{2(k-1)(1-2a^2)^{k-1}},
$$
and we can specify $\mu_0$ for any fixed $k>1$ similarly to $k=3/2$.
}
\end{remark}
\begin{remark}\label{R:1.38}
{\rm
Inequalities (\ref{ineqS2}) and (\ref{ineqT2}) hold for
$k=1.38\dots$.
}
\end{remark}

\begin{figure}[htb]
\centerline{\psfig{file=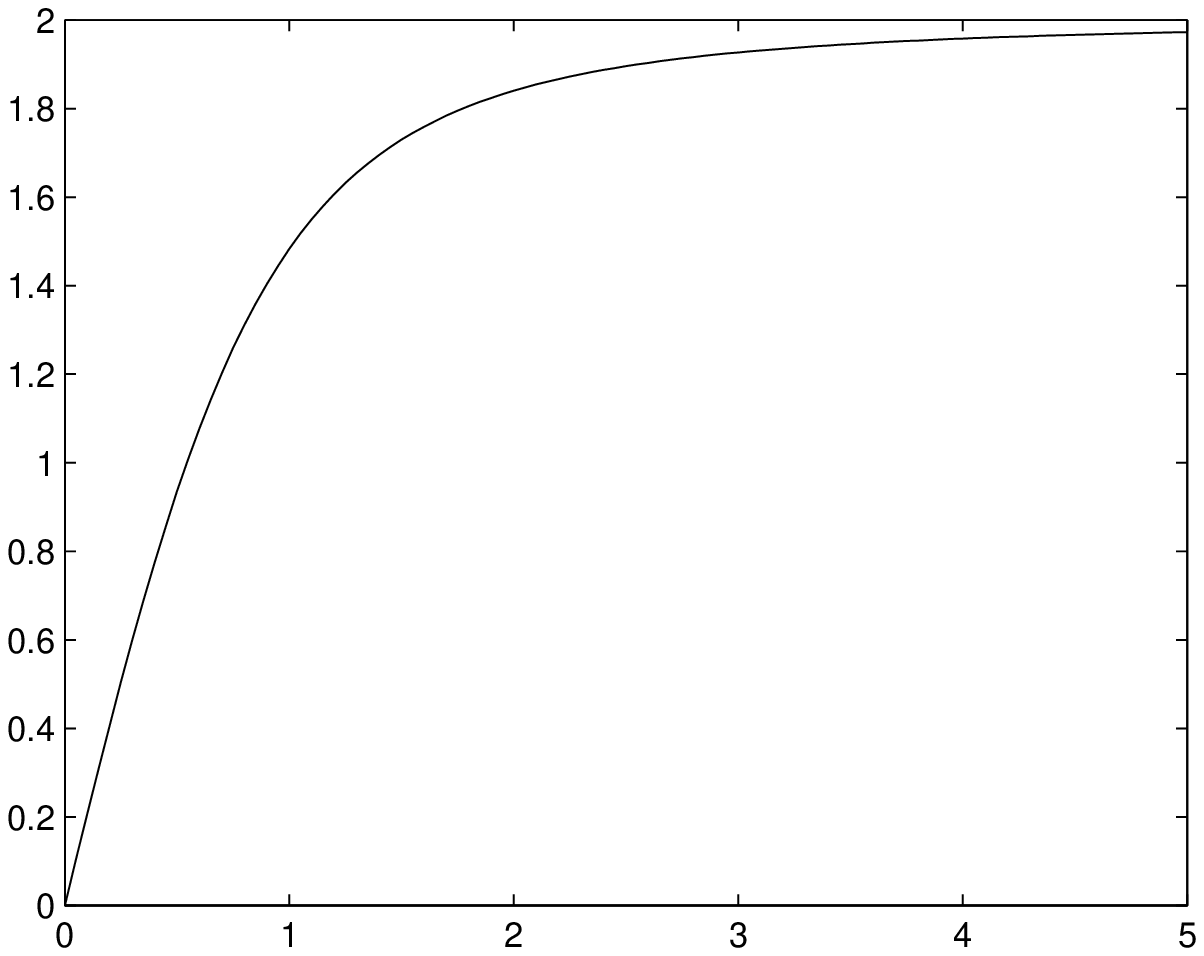,width=6.5cm,height=5cm,angle=0}
\psfig{file=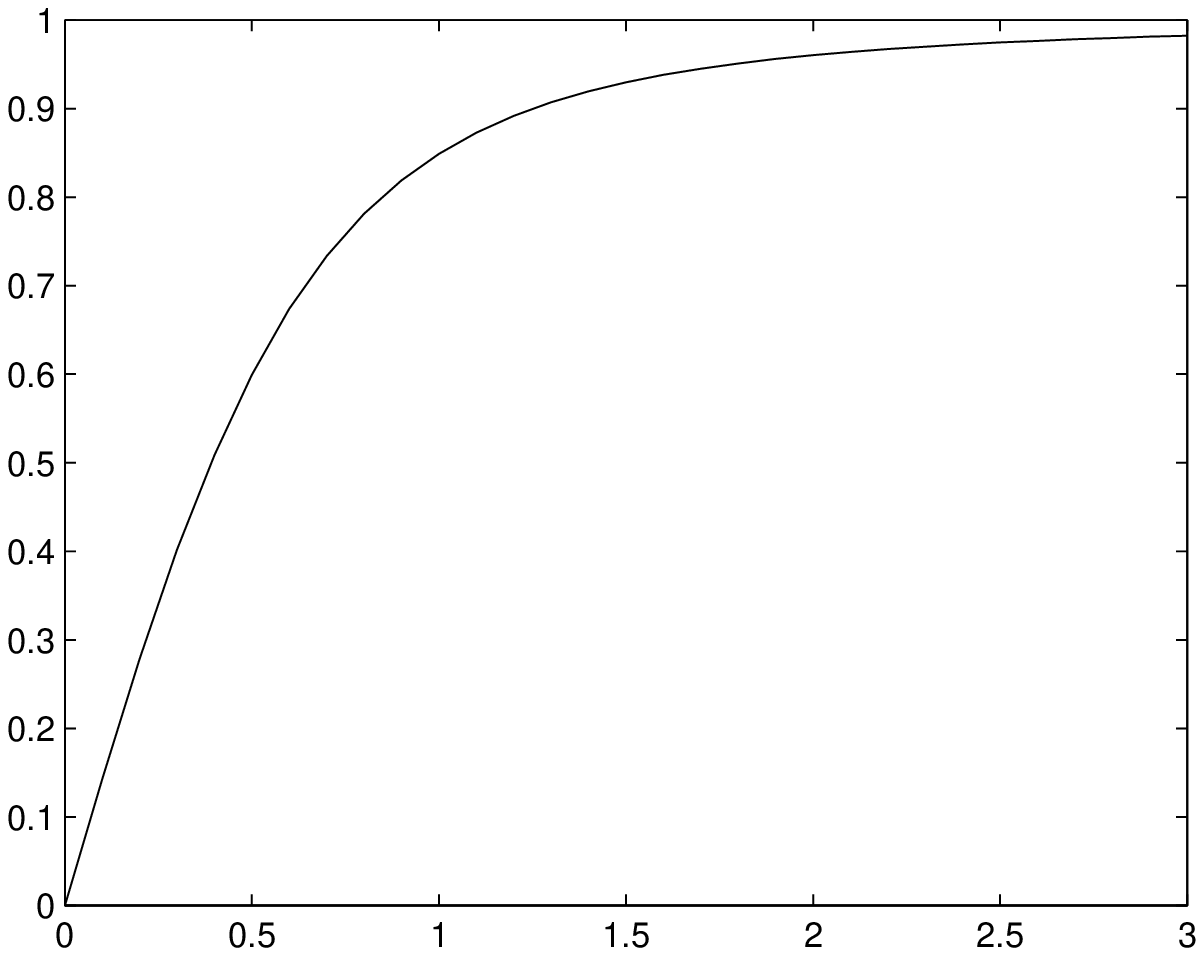,width=6.5cm,height=5cm,angle=0} }
\caption{Graphs of the functions $H(\mu)$ and
$\frac{k-1}\pi F(\mu)$ on the corresponding intervals
$[0,\mu_0]$ for $k=3/2$.} \label{fig:S2}
\end{figure}

\begin{proposition}\label{P:T3S3}
The following inequalities hold for $\mu\ge0$:
\begin{equation}\label{ineqT3S3}
\aligned
H_{\mathbb{S}^3}(\mu)&:= \mu\sum_{n=1}^\infty
\frac{(n+1)^2}{(n(n+2)+\mu^2)^2}\le
\delta_{\mathbb{S}^3}\int_0^\infty \frac{r^2dr}{(r^2+1)^2}=
\delta_{\mathbb{S}^3}\cdot\frac\pi 4,\\
F_{\mathbb{T}^3}(\mu)&:=\mu\sum_{m\in\mathbb{Z}^3_0}
\frac{1}{(|m|^2+\mu^2)^2}<\delta_{\mathbb{T}^3}\int_{\mathbb{R}^3}
\frac{dx}{(x^2+1)^2}=\delta_{\mathbb{T}^3}\cdot\pi^2,
\endaligned
\end{equation}
where $\delta_{\mathbb{S}^3}=1.0139\dots$ and
$\delta_{\mathbb{T}^3}=1$.
\end{proposition}
\begin{proof} Calculations show that the function
$H_{\mathbb{S}^3}(\mu)$ attains a global maximum
at $\mu_*=3.312\dots$, which is $1.0139\ldots=:\delta_{\mathbb{S}^3}$
times greater than $H_{\mathbb{S}^3}(\infty)=\pi/4$.
In calculations we can also take
advantage of the fact that for  $H_{\mathbb{S}^3}(\mu)$ there exists an
explicit formula. In fact, using the formula
$$
\sum_{n=1}^\infty
\frac{n^2}{(n^2+\nu^2)^2}=\frac\pi 4\frac{\coth(\pi
\nu)}\nu+
\frac{\pi^2}4\left(1-\coth^2(\pi\nu)\right),
$$
and noting that $n(n+2)=(n+1)^2-1$ we see that
$H_{\mathbb{S}^3}(\mu)$ is equal to
$$
\frac\pi 4\frac\mu{\sqrt{\mu^2-1}}\coth(\pi
\sqrt{\mu^2-1})+
\frac{\pi^2\mu}4\left(1-\coth^2(\pi\sqrt{\mu^2-1})\right)-
\frac1{\mu^3}.
$$
Unlike the 2D case,  for large $\mu$,
$H_{\mathbb{S}^3}(\mu)>H_{\mathbb{S}^3}(\infty)=\pi/4$.

For the second sum the Poisson summation formula
and~(\ref{Four_Stein}) give
$$
F_{\mathbb{T}^3}(\mu)=
\pi^2-\frac1{\mu^3}+\pi^2\sum_{m\in\mathbb{Z}^3_0}e^{-2\pi\mu|m|}=
\pi^2-\frac1{\mu^3}+O(e^{-C\mu}).
$$
We  find a $\mu_0$ such that $F_{\mathbb{T}^3}(\mu)<
\pi^2$ on $[\mu_0,\infty)$ and then verify the inequality
on the remaining  finite interval $[0,\mu_0]$ by calculations.
We omit the details concerning $\mu_0$ that
are similar to those in Proposition~\ref{P:T2}. The graphs
of $H_{\mathbb{S}^3}(\mu)$ and $F_{\mathbb{T}^3}(\mu)$ are
shown in Fig.~\ref{fig:S3}.
\end{proof}
\begin{figure}[htb]
\centerline{\psfig{file=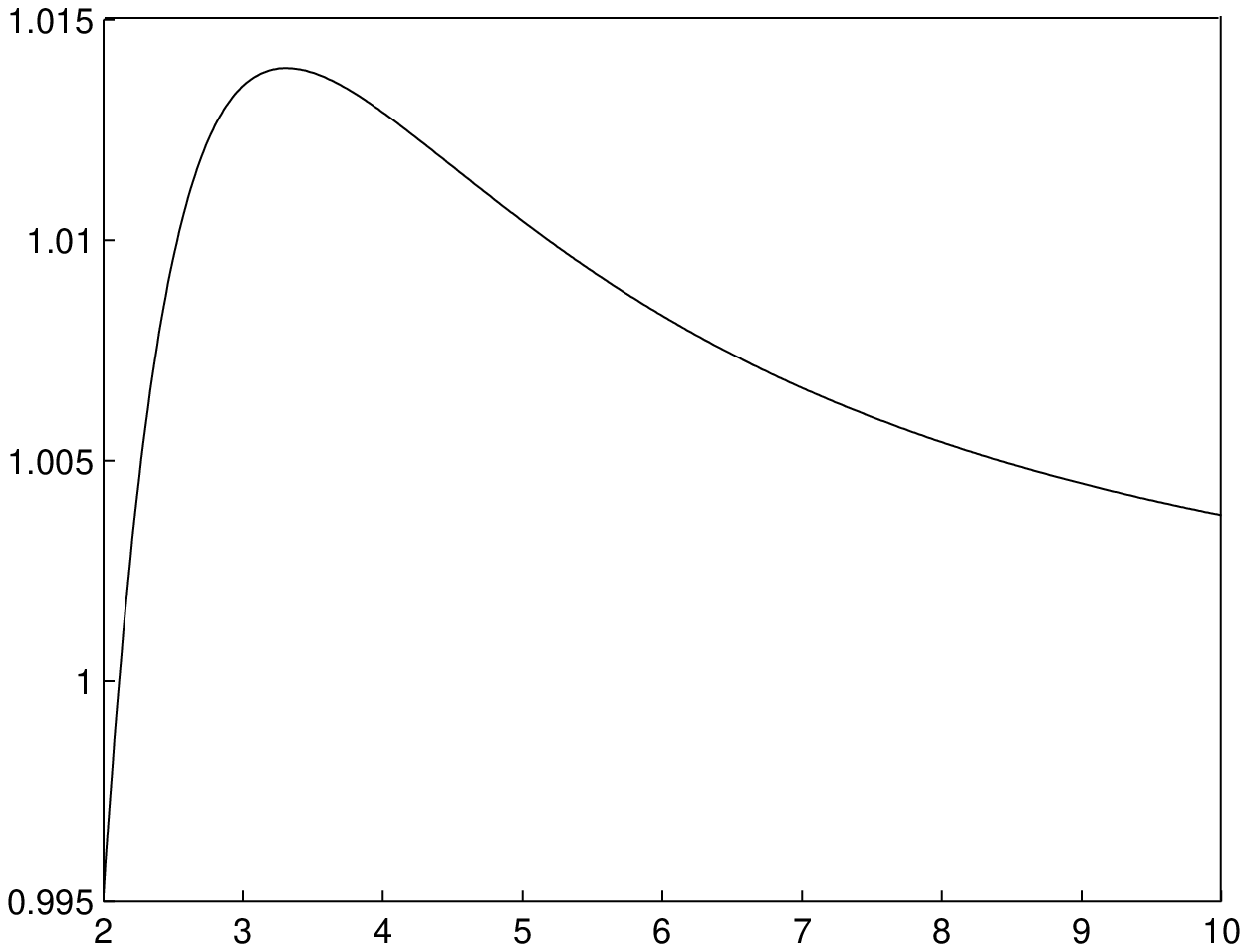,width=6.5cm,height=5cm,angle=0}
\psfig{file=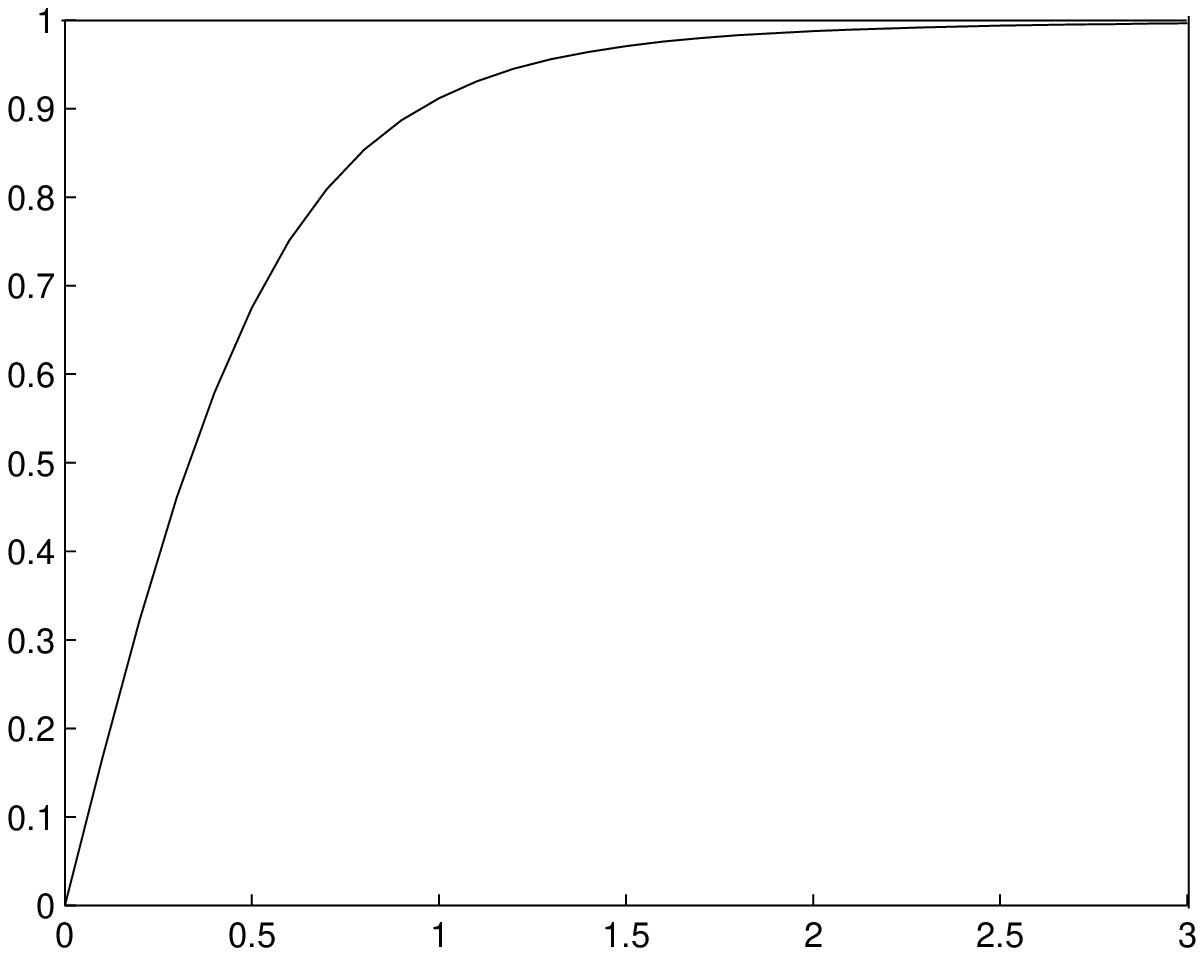,width=6.5cm,height=5cm,angle=0} }
\caption{Graphs of the functions $\frac4\pi
H_{\mathbb{S}^3}(\mu)$ and
$\frac1{\pi^2}F_{\mathbb{T}^3}(\mu)$.} \label{fig:S3}
\end{figure}

\setcounter{equation}{0}
\section{Appendix. Estimate of $\mu_0$ for the sphere}\label{S:Appen}
\begin{lemma}\label{A:mu0}
For $k=3/2$ inequality~(\ref{ineqS2}) holds for
$\mu\in[\mu_0,\infty)$, where $\mu_0=5.0833$.
\end{lemma}
\begin{proof} It follows from~(\ref{int-sum1}) that we have
to show that for $f(x)=1/(x+1)^k$ and $\mu\ge\mu_0$
\begin{equation}\label{int-sum}
\int_{a_1}^{a_2}f(x)dx- a_2f(a_2)/2\,>\, -\sum_{n=2}^\infty
R_{a_n,a_{n+1}}(f),
\end{equation}
the main task being specifying $\mu_0$. Since $f(x)$ is
monotone decreasing,
 $\int_{a_1}^{a_2}f(x)dx>a_2f(a_2)$, and
 the left-hand side is greater than
 \begin{equation}\label{int-sum-left}
\frac1{\mu^2}\frac1{(1+\frac2{\mu^2})^k}>
\frac1{\mu^2}\left(1-\frac{2k}{\mu^2}\right)=
t-2kt^2=:L_k(t),\ t=\mu^{-2}.
\end{equation}
For the right-hand side of~(\ref{int-sum}) with
$f''(x)=k(k+1)/(x+1)^{k+2}$  and
$\xi$ in~(\ref{ksi}) satisfying $ \xi>n(n-1)/\mu^2>((n-1)/\mu)^2$ we have
\begin{equation}
\aligned -\sum_{n=2}^\infty R_{a_n,a_{n+1}}(f)=
\frac{2k(k+1)}{3\mu^2}\frac1\mu\sum_{n=2}^\infty
\frac{(n/\mu)^3}{(\xi_n+1)^{k+2}},
\endaligned
\end{equation}
and
$$
\aligned
 \frac1\mu\sum_{n=2}^\infty
\frac{(n/\mu)^3}{(\xi_n+1)^{k+2}}<\frac1\mu\sum_{n=1}^\infty
\frac{((n+1)/\mu)^3}{((n/\mu)^2+1)^{k+2}}=\\=
\frac1\mu\sum_{n=1}^\infty
g_1(n/\mu)+\frac3{\mu^2}\sum_{n=1}^\infty
g_2(n/\mu)+\frac3{\mu^3}\sum_{n=1}^\infty
g_3(n/\mu)+\frac1{\mu^4}\sum_{n=1}^\infty g_4(n/\mu),
\endaligned
$$
where
$
g_j(x)=\frac{x^{4-j}}{(x^2+1)^{k+2}}
$, $j=1,2,3,4$.
 The function $g_1(x)$ has a unique global maximum attained at
$
x_0=\left(\frac3{2k+1}\right)^{1/2}.
$
Therefore
$$
\aligned
 \frac1\mu\sum_{n=1}^\infty
g_1(n/\mu)< x_0g_1(x_0)+\int_{x_0}^\infty g_1(x)dx=\\=
\frac{9(2k+1)^k}{(2k+4)^{k+2}}+\frac1{2k(k+1)}
\frac{(5k+4)(2k+1)^k}{(2k+4)^{k+1}}=:G_1(k).
\endaligned
$$
Similarly (replacing $x_0$ in the integral by $0$)
$$
\aligned
\frac1\mu\sum_{n=1}^\infty g_2(n/\mu)<
\frac{(k+1)^{k+1/2}}{(k+2)^{k+2}}+\frac12
\frac{\Gamma(3/2)\Gamma(k+1/2)}{\Gamma(k+2)}=:G_2(k),\\
\frac1\mu\sum_{n=1}^\infty g_3(n/\mu)<
\frac{(2k+3)^{k+3/2}}{(2k+4)^{k+2}}+\frac12
\frac{1}{k+1}=:G_3(k),\\
\frac1\mu\sum_{n=1}^\infty g_4(n/\mu)<
\frac12
\frac{\Gamma(1/2)\Gamma(k+3/2)}{\Gamma(k+2)}=:G_4(k).
\endaligned
$$
which gives that the right-hand side in~(\ref{int-sum}) is
less than
$$
\frac{2k(k+1)}{3}\left(G_1(k)t+3G_2(k)t^{3/2}+3G_3(k)t^2+G_1(k)t^{5/2}
\right)=:R_k(t)
$$
and
$ R_{3/2}(t)=0.5317\cdot t+1.5844\cdot t^{3/2}+
3.2851\cdot t^2+1.3333\cdot t^{5/2}$.
Obviously, $L_{3/2}(t)=t-3t^2\ge R_{3/2}(t)$ for $t\in[0,t_0]$,
where $t_0$ is the first root of the equation
$L_{3/2}(t)- R_{3/2}(t)=0$.
We find that $t_0=0.0387$. Accordingly, (\ref{int-sum})
holds for all $\mu\ge\mu_0=(1/t_0)^{1/2}=5.0833$.
Explicitly calculating the integral on the left-hand side of
(\ref{int-sum}) and estimating the series involving
$g_2$ and $g_3$ in the same way as $g_1$ we
have $R_{3/2}(t)=0.5317\cdot t+0.90074\cdot t^{3/2}+ 2.8054\cdot t^2+
1.3333\cdot t^{5/2}$  and therefore can
 improve the estimate: $\mu_0=3.9229$.
\end{proof}

\subsection*{Acknowledgments}
The author thanks S.V.\,Zelik for helpful discussions.
 This work was supported  by  the RFBR  grants no.~09-01-00288,
 no.~11-01-00339 and by  the RAS Programme no.1.

\bibliographystyle{amsplain}

\begin{thebibliography}{99}

\bibitem{A}\textrm{H.\,Araki, }
\textrm{On an inequality of Lieb and Thirring.}
\textit{Lett. Math. Phys.} \textbf{19} (1990), 167--170.

\bibitem{AL}\textrm{M.\,Aizenman and  E.H.\,Lieb,}
\textrm{On semi-classical bounds for eigenvalues of
Schr\"odinger operators.} \textit{Phys. Lett.} \textbf{66A}
(1978), 427--429.


\bibitem{B-V}\textrm {A.V.\,Babin and M.I.\,Vishik,}
\textit {Attractors of Evolution Equations}. \textrm Nauka,
\textrm Moscow, \textrm 1989; \textrm {English transl.}
\textrm{North-Holland, Amsterdam}, \textrm{1992}.



\bibitem{Zelik}\textrm{M.V.\,Bartuccelli, J.\,Deane, and S.V.\,Zelik,}
\textrm{Asymptotic expansions and extremals for the
critical Sobolev and Gagliardo-Nirenberg inequalities on a
torus.} arXiv:1012.2061 (2010).

\bibitem{B-L}\textrm{R.\,Benguria and M.\,Loss,}
\textrm{A simple proof of a theorem by Laptev and Weidl.}
\textit{Math. Res. Lett.}
\textbf{7} (2000), 195--203.

\bibitem{Ch-V-book}\textrm{V.V.\,Chepyzhov and M.I.\,Vishik,}
\textit{Attractors for Equations of Mathematical Physics}.
\textrm{Providence, RI, Amer. Math. Soc.}, 2002.
%\textrm{(Amer. Math. Soc. Colloq. Publ. V.49.)}


\bibitem{CF88}\textrm{P.\,Constantin and  C.\,Foias,}
\textit{Navier-Stokes Equations}.
\textrm{The University of Chicago Press, Chicago}, 1988.


\bibitem{D-L-L}\textrm{J.\,Dolbeault, A.\,Laptev, and M.\,Loss,}
\textrm{Lieb--Thirring inequalities with improved
constants.}
\textit{J. European Math. Soc.}
\textbf{10}:4 (2008), 1121--1126.

\bibitem{E-F}\textrm{A.\,Eden and C.\,Foias,}
\textrm{ A simple proof of the generalized Lieb--Thirring inequalities in one
space dimension}.
\textit{J. Math. Anal. Appl.} \textbf{162}
(1991), 250--254.


\bibitem{G-M-T}\textrm{J. M.\,Ghidaglia, M.\,Marion, and R.\,Temam,}
\textrm{Generalization of the Sobolev--Lieb--Thirring
inequalities
 and applications to the dimension of attractors}.
\textit{Differential and Integral Equations} \textbf{1}:1  (1988), 1--21.

\bibitem{H-L-W} \textrm{D.\,Hundertmark, A.\,Laptev, and  T.\,Weidl,}
\textrm{ New bounds on the Lieb--Thirring constants}.
\textit{ Inv. Math.} \textbf{140} (2000), 3, 693--704.



\bibitem{ILMS93} \textrm{A.A.\,Ilyin,}
\textrm{Lieb--Thirring inequalities on the $N$-sphere and
in the plane, and some applications.} \textit{ Proc. London
Math. Soc.} \textbf{67} (1993), 159--182.


\bibitem{I93} \textrm{A.A.\,Ilyin,}
\textrm{ Partly dissipative semigroups generated by the
Navier--Stokes system on two-dimensional manifolds and
their attractors.} \textit{ Mat. Sbornik} \textbf{184}:1,
55--88 (1993) \textrm{English transl. in} \textit{ Russ.
Acad. Sci. Sb. Math.} \textbf{78}:1, 47--76 (1993).


\bibitem{I98JLMS}\textrm{A.A.\,Ilyin,}
\textrm{ Best constants in multiplicative inequalities for
sup-norms.}
 {\it J. London Math. Soc.}(2)
 \textbf{58}, 84--96 (1998).


\bibitem{I-M-T} \textrm{A.A.\,Ilyin, A.\,Miranville, and E.S.\,Titi,}
\textrm{A small viscosity sharp  estimate for the global
attractor of  the 2-D damped-driven Navier--Stokes
equations.} \textit{Commun. Math. Sciences} \textbf{2}:3  (2004),
403--426.

\bibitem{Kashin} \textrm{B.S.\,Kashin,}
\textrm{On a class of inequalities for orthonormal systems.}
\textit{Mat. Zametki}
\textbf{80}:2 (2006), 204--208;
English transl.
\textit{Math. Notes}
\textbf{80} (2006), 199--203.

\bibitem{Krylov}\textrm{V.I.\,Krylov,}
\textit{Approximate calculation of integrals.}
\textrm{Gos. Izdat. Fiz.--Mat. Lit.,}
\textrm{Moscow}, 1959;
English transl. Macmillan, New York, 1962.

\bibitem{Lap-Weid}\textrm{A.\,Laptev and T.\,Weidl,}
{\textrm Sharp Lieb--Thirring inequalities in high dimensions.}
\textit{Acta Math.} \textbf{184}  (2000), 87--111.




\bibitem{Lieb}\textrm{E.\,Lieb,}
 {\textrm On characteristic exponents in turbulence},
 \textit{Commun. Math. Phys.}
 \textbf{92}   (1984), 473--480.


\bibitem{LT} \textrm{E.\,Lieb and W.\,Thirring,}
\textrm{ Inequalities for the moments of the
eigenvalues of the Schr\"o\-dinger Hamiltonian and their relation to
Sobolev inequalities, Studies in Mathematical Physics. Essays in
honor  of Valentine Bargmann,}
\textrm{ Princeton University Press},
 \textrm{ Princeton NJ}, 269--303 (1976).

\bibitem{S-W} \textrm{E.M.\,Stein and G.\,Weiss,}
\textit{ Introduction to Fourier analysis on Euclidean spaces.}
 \textrm{Princeton University Press},
 \textrm{ Princeton NJ}, 1972.

\bibitem{Taikov}\textrm{L.V.\,Taikov,}
\textrm{Kolmogorov-type inequalities and the best formulas
for numerical differentiation.}
\textit {Mat. Zametki} \textbf{4},  233--238 (1968);
\textrm{ English transl.}
\textit {Math. Notes} \textbf{4}  (1968), 631--634.

\bibitem{T}\textrm{R.\,Temam,}
\textit{Infinite Dimensional Dynamical Systems in
Mechanics and Physics, \rm 2nd ed.},
\textrm{Sprin\-ger-Ver\-lag}, \textrm{New York}, 1997.

\bibitem{Weidl}\textrm{T.\,Weidl,}
 {\textrm On the Lieb--Thirring constants $\mathrm{L}_{\gamma,1}$ for
 $\gamma\ge1/2$.}
 \textit{Commun. Math. Phys.}
 \textbf{178}:1, 135--146
  (1996).
\end{thebibliography}

\end{document}